\documentclass{amsart}
\newtheorem{teo} {Theorem}
\newtheorem{cor} {Corollary}

\begin{document}

\title{A note on the Li\'{e}nard-Chipart criterion and roots of 
some families of polynomials}

\author{Renato B. Bortolatto}
\thanks{Universidade Tecnol\'{o}gica Federal do Paran\'{a} (UTFPR). 
Campus Londrina - PR - Brazil.}

\begin{abstract}
We present some inequalities that provide different sufficient conditions for 
an univariate monic polynomial to be Hurwitz unstable.
These are motivated by difficult control problems where direct 
application of the Li\'{e}nard-Chipart criterion is not feasible. 
Hurwitz stability of some polynomials of degree five is also discussed. 
These results may be interpreted as stability results for some interval 
polynomials.
\end{abstract}



\maketitle

\section{Introduction}

In this note we'll say that a polynomial 
$$p(s) = s^n + \alpha_1 s^{n -1} + \ldots + \alpha_n$$
is Hurwitz stable 
if all roots have negative real part. Otherwise, the 
polynomial will be said Hurwitz unstable.

When $\alpha_i$ is positive for $i = 1, \ldots, n$, that is, 
all coefficients of $p(s)$ are positive, it is 
straight forward to prove that no real root can be strictly positive: 
Actually, it suffices to see that if $s > 0$ then $ p(s) > 0$, 
but this observation also follows from Descartes' rule of signs. 
Furthermore if $p(s)$ has a root in $s = 0$ then $\alpha_n = 0$. 

Hence, if we're trying to prove stability of a polynomial and if all 
coefficients 
$\alpha_i$ of $p(s)$ are strictly positive, we are left only to worry about the 
possibility of one of the complex roots to have positive real part. This 
still is a very complicated problem. For instance, bounds for roots like 
the ones derived from well-known ideas in complex analysis are not 
immediately useful since, a priori, all restrictions are given in absolute 
value (a non-trivial result that can be obtained using complex 
analysis is given by Routh's algorithm, see \cite{G}).

The Li\'{e}nard-Chipart stability criterion is a standard tool to understand   
the Hurwitz stability problem, which in turn has important consequences on 
the dynamics of some systems of differential equations. 
The theorem can be enunciated in the following way:

\begin{teo}[see \cite{G}, \cite{LC}] \label{lien}
Let $p(s) = s^n + \alpha_1 s^{n -1} + \ldots + \alpha_n$. 
A necessary and sufficient condition for all 
roots of $p(z)$ to have negative real parts 
is that 
 
$$\alpha_{i} > 0, \forall i =1, 2, \ldots, n \; \text{ and } \; 
\Delta_2 > 0, \Delta_4 > 0, \ldots $$
\end{teo}
where 
\begin{equation*} 
\Delta_i = \begin{vmatrix} 
\alpha_1 & \alpha_3 & \alpha_5 & \ldots \\
1        & \alpha_2 & \alpha_4 & \ldots \\
0        & \alpha_1 & \alpha_3 & \ldots \\
0        &    1     & \alpha_2 & \ldots \\ 
0        &    0     & \alpha_1 & \ldots \\
\vdots   &    0   &  1  &       \\
\vdots   & \vdots   &    & & \\
0        &    0     &  \ldots  & \ldots & \alpha_i 
\end{vmatrix}
\end{equation*} 
with $\alpha_k = 0$ if $k > n$. 

This result can be thought as a simplification of the Routh-Hurwitz theorem 
in which, if all coefficients of $p(s)$ are positive, stability can be 
guaranteed by checking that $\Delta_i$ is positive for $i = 1, \ldots, n$. 
The Li\'{e}nard-Chipart criterion can also be stated using $\Delta_i$ for 
$i$ odd instead of even.

Although linearization together with the Li\'{e}nard-Chipart criterion may 
reduce the dynamic stability problem to a calculation (see \cite{K}), 
control problems in 
engineering can generate, in this manner, very 
complicated algebraic expressions 
that resist to simplifications, even by means of computational 
algebra (see for instance \cite{J}). 
Numerical simulations are often the only solution available to study 
the dynamics, as we resort to sampling values for the coefficients. On the 
other hand generic conditions on the coefficients, like the 
aforementioned $\alpha_i > 0$, may be physically natural, easier to evaluate 
or even be chosen by design. 

What we aim to do in this note is simplify a condition in the
Li\'{e}nard-Chipart criterion as much as possible, to provide more 
approachable and non-trivial necessary or sufficient conditions for the Hurwitz 
(in)stability of some polynomials. 
With this in mind, and assuming that all $\alpha_i$ are positive, 
it is worthwhile to study the sign of 
$\Delta_2$ instead of 
$\Delta_3$, so the next term to study is $\Delta_4$, with 
the sign of $\Delta_6$ being apparently much harder to understand. 

To study the sign of $\Delta_4$ we use a,
to the best of our knowledge, new formula for $\Delta_4$
which is presented in Theorem \ref{main} in the next section. 


\section{The main result}

\begin{teo} \label{main} We have that 
$$ \Delta_ 4 = -\alpha_2 (\alpha_5 - \alpha_1\alpha_4) 
\Delta_2 - \alpha_4 \Delta_2^2 
- (\alpha_5 - \alpha_1 \alpha_4)^2 - (\alpha_7 - \alpha_1\alpha_6)\Delta_2 $$ 
where $\Delta_2 = \alpha_1\alpha_2 - \alpha_3$.
\end{teo}

\begin{proof}
By definition  
$$\Delta_4 = 
\begin{vmatrix} 
\alpha_1 & \alpha_3 & \alpha_5 & \alpha_7 \\ 
1        & \alpha_2 & \alpha_4 & \alpha_6 \\ 
0		 & \alpha_1 & \alpha_3 & \alpha_5 \\ 
0 		 & 1		& \alpha_2 & \alpha_4 \\
\end{vmatrix}
$$
Expansion by the first column leads to the expression 
\begin{multline*} 
\Delta_4 = \alpha_1\alpha_2\alpha_3\alpha_4 + 2\alpha_1\alpha_4\alpha_5 
- \alpha_1\alpha_2^2\alpha_5 - \alpha_1^2 \alpha_4^2 - 
\alpha_3^2 \alpha_4 - \alpha_5^2 + \alpha_2\alpha_3\alpha_5 + \\
+ \alpha_1^2\alpha_2\alpha_6 - \alpha_1\alpha_3\alpha_6 
- \alpha_1\alpha_2\alpha_7 + \alpha_3\alpha_7 
\end{multline*}
that can be obtained by developing the formula given in the statement.
\end{proof}

The alternative formula for $\Delta_4$ in Theorem \ref{main} 
was originally obtained in most part 
by the method described in \cite{B} for degree five  
(where $\alpha_6$ and $\alpha_7$ are both zero). 

This formula facilitates the determination of the sign of $\Delta_4$, 
as we'll see bellow, and can also be easily implemented to minimize 
computational cost and cumulative error of non-specialized software, 
being particularly useful to study polynomials of degree five. 


\section{Some consequences for Hurwitz instability}

\begin{cor} \label{cor1}
Let $p(s) = s^n + \alpha_1 s^{n -1} + \ldots + \alpha_n$, with $n \in 
\mathbb{N}$ greater than or equal to $5$. 
Then a sufficient condition for $p(s)$ to be unstable is that 
\begin{equation*} 
\alpha_5 - \alpha_1 \alpha_4 \geq 0 \text{  \qquad and \qquad  } 
\alpha_7 - \alpha_1\alpha_6 \geq 0
\end{equation*}
\end{cor}

\begin{proof} 
We can assume that $\alpha_i> 0$, for $i = 1, \ldots, n$ 
and $\Delta_2 > 0$, since otherwise 
$p(s)$ is automatically unstable by the Li\'{e}nard-Chipart criterion 
(Theorem \ref{lien}). 
Therefore, by Theorem \ref{main}, 
$\Delta_4 \leq 0$ so $p(s)$ is Hurwitz unstable.
\end{proof}

Note that the second inequality in the corollary is automatically 
satisfied if the degree of $p(s)$ is five. 

\begin{cor} \label{cor2}
If $\alpha_7 - \alpha_1\alpha_6 \geq 0$ and 
$n$ is greater than or equal to $5$ the polynomial 
$$p(s) = s^n + \alpha_1 s^{n - 1} + 2 s^{n - 2} 
+ \alpha_3 s^{n - 3} + s^{n - 4} + \ldots + \alpha_n $$ 
is unstable 
\end{cor}

\begin{proof} Assume that $\Delta_2 > 0$. 
Since $\alpha_2 = 2 $ and $\alpha_4 = 1$ we have that 
\begin{align*} \Delta_ 4 & = - 2 (\alpha_5 - \alpha_1\alpha_4)\Delta_2 
- \Delta_2^2 - (\alpha_5 - \alpha_1 \alpha_4)^2 
- (\alpha_7 - \alpha_1\alpha_6)\Delta_2 = \\ 
& = - [ (\alpha_1\alpha_2 - \alpha_3) 
+ (\alpha_5 - \alpha_1 \alpha_4) ]^2 - (\alpha_7 - \alpha_1\alpha_6)\Delta_2 
\leq 0
\end{align*}
\end{proof}

Recall that if every $\alpha_i$ is strictly positive then 
all real roots must be strictly negative. Therefore in such a case and if
$p(s)$ is as in Corollary \ref{cor2} then
$p(s)$ needs to have at least a pair of complex conjugated 
roots with positive real part.  Also, if $n$ is odd we can assure that there 
will be at least one negative real root.

\begin{cor} \label{cor3}
Let $p(s) = s^n + \alpha_1 s^{n -1} + \ldots + \alpha_n$, with $n \in 
\mathbb{N}$ greater than or equal to $5$. 
Then a sufficient condition for $p(s)$ to be unstable is that 
$$ \alpha_2^2 - 4 \alpha_4 \leq 0 \text{ \qquad and \qquad} 
\alpha_7 - \alpha_1\alpha_6 \geq 0$$
\end{cor}

\begin{proof} Note that 
if $\alpha_5 - \alpha_1\alpha_4 = 0$ then $\Delta_4 < 0$ so there's nothing 
to do.Otherwise define
$$\Gamma := \frac{\Delta_4 + (\alpha_7 - \alpha_1\alpha_6)\Delta_2}{(\alpha_5 - \alpha_1\alpha_4)^2} =   
-\alpha_2 \frac{(\alpha_1\alpha_2 - \alpha_3)}{(\alpha_5 - \alpha_1\alpha_4)}
- \alpha_4 
\frac{(\alpha_1\alpha_2 - \alpha_3)^2}{(\alpha_5 - \alpha_1\alpha_4)^2} - 1
$$
and note that $\Gamma < 0$ implies that $\Delta_4 < 0$. Let 
$$ \gamma :=  
\frac{(\alpha_1\alpha_2 - \alpha_3)}{(\alpha_5 - \alpha_1\alpha_4)} $$ 
so $\Gamma = - \alpha_4 \gamma^2 -\alpha_2 \gamma - 1$. 

As a function of $\gamma$, $\Gamma$ is concave, since we can assume 
that $\alpha_4 > 0$. 
Then for $\Gamma(\gamma)$ to be positive for some $\gamma$ we need to have that 
$\alpha_2^2 - 4 \alpha_4 > 0$ and that 
$$ \frac{\alpha_2 + \sqrt{\alpha_2^2 - 4 \alpha_4}}{-2\alpha_4} < \gamma 
< \frac{\alpha_2 - \sqrt{\alpha_2^2 - 4 \alpha_4}}{-2\alpha_4} $$
 
\end{proof}

An easy consequence of the previous result is that, if 
$\alpha_7 - \alpha_1\alpha_6 \geq 0$,
$$p(s) = s^n + \alpha_1 s^{n - 1} + \alpha_2 s^{n - 2} 
+ \alpha_3 s^{n - 3} + s^{n - 4} + \ldots + \alpha_n $$ 
is unstable for every $0 < \alpha_2 \leq 2$ 
(compare to Corollary \ref{cor2}). 


\section{Some consequences for Hurwitz stability}

A reformulation of the previous corollary 
gives a criteria for stability. 

\begin{cor} \label{cor4}
Let $p(s) = s^5 + \alpha_1 s^{4} + \alpha_2 s^{3} 
+ \alpha_3 s^{2} + \alpha_4 s + \alpha_5 $. 
Assume that $\alpha_i > 0 $ for $i = 1, \ldots, 5$ and that $\Delta_2 > 0$.  
Then a necessary condition for $p(s)$ to be Hurwitz stable is that 
$$ \alpha_2^2 - 4 \alpha_4 > 0 $$
\end{cor} 

Note that the hypotheses of the former corollary implies in particular 
that $ (\alpha_5  - \alpha_1\alpha_4) < 0$, since otherwise 
$\Gamma(\gamma) < 0$ (see Corollary \ref{cor1}). 
Compare this to how the Li\'{e}nard-Chipart shows 
that some conditions of the Routh-Hurwitz theorem are not independent. 
Our final result
is still impractical as a tool to study problems like in \cite{J}, 
but is provided as an alternative 
to direct application of Theorem \ref{lien} or to 
the final inequality in Corollary \ref{cor3}.  

\begin{cor} \label{cor5}
Let $p(s) = s^5 + \alpha_1 s^4 + \alpha_2 s^3 + \alpha_3 s^2 + 
\alpha_4 s + \alpha_5 $ where $\alpha_i > 0 $ for $i = 1, 2, 3, 4, 5$. 
Assume that $ \Delta_2 > 0 $ and assume that 
$$ \alpha_2^2 - 4 \alpha_4 > 0 $$
Then a sufficient condition for $p(s)$ to be stable is that 
$$ \frac{\alpha_1\alpha_2 - \alpha_3}{\alpha_5 - \alpha_1 \alpha_4} 
= \frac{- \alpha_2}{2\alpha_4} $$
\end{cor}

\begin{proof} 
Just note that if $\overline{\gamma}$ denotes de vertex of $\Gamma(\gamma)$ 
then $\overline{\gamma} := \frac{- \alpha_2}{2\alpha_4} $. 
Since $\Gamma(\gamma)$ has two distinct roots then 
$\Gamma (\gamma) = \Gamma (\overline{\gamma}) > 0$, therefore $\Delta_4 > 0$. 

It's also possible to check directly that $ \Delta_4 > 0 $ by studying the 
sign of $ \Gamma $.
\end{proof}

Finally, we note that 
the sufficient condition above is equivalent to 
$$ \Delta_2 = \alpha_3 - \frac{\alpha_2 \alpha_5}{\alpha_4} \text {\qquad 
 and to \qquad } 
\alpha_1 = \frac{2 \alpha_3}{\alpha_2} - \frac{\alpha_5}{\alpha_4}$$


\end{document}